\documentclass{amsart}

\newcommand{\beqa}{\begin{eqnarray*}}
\newcommand{\eeqa}{\end{eqnarray*}}
\newcommand{\beqn}{\begin{eqnarray}}
\newcommand{\eeqn}{\end{eqnarray}}

\newcommand{\bT}{\mathbb T}

\newcommand{\ov}{\overline}

\newcommand{\al}{\alpha}
\newcommand{\be}{\beta}
\newcommand{\ga}{\gamma}

\newcommand{\la}{\lambda}

\newcounter{cnt1}
\newcounter{cnt2}
\newcounter{cnt3}
\newcommand{\blr}{\begin{list}{$($\roman{cnt1}$)$}
 {\usecounter{cnt1} \setlength{\topsep}{0pt}
 \setlength{\itemsep}{0pt}}}
\newcommand{\bla}{\begin{list}{$($\alph{cnt2}$)$}
 {\usecounter{cnt2} \setlength{\topsep}{0pt}
 \setlength{\itemsep}{0pt}}}
\newcommand{\bln}{\begin{list}{$($\arabic{cnt3}$)$}
 {\usecounter{cnt3} \setlength{\topsep}{0pt}
 \setlength{\itemsep}{0pt}}}
\newcommand{\el}{\end{list}}

\newtheorem{thm}{Theorem}[section]
\newtheorem{lem}[thm]{Lemma}

\newtheorem{Def}[thm]{Definition}

\newtheorem{rem}[thm]{Remark}
\newcommand{\Rem}{\begin{rem} \rm}
\newcommand{\bdfn}{\begin{Def} \rm}
\newcommand{\edfn}{\end{Def}}

\newcommand{\ba}{\begin{array}}
\newcommand{\ea}{\end{array}}

\begin{document}
\sloppy

\title[Generalized 3-circular projections]
{Generalized 3-circular projections in some Banach spaces}

\author[Abubaker]{A. B. Abubaker}
\address[Abdullah Bin Abubaker]{Department of Mathematics and Statistics\\
Indian Institute of Technology Kanpur \\
India, \textit{E-mail~:} \textit{abdullah@iitk.ac.in}}

\author[S Dutta]{S.\ Dutta}
\address[S Dutta]{Department of Mathematics and Statistics\\
Indian Institute of Technology Kanpur \\
India, \textit{E-mail~:} \textit{sudipta@iitk.ac.in}}

.


\begin{abstract}
Recently in a series of papers it is observed that in many Banach
spaces, which include classical spaces $C(\Omega)$ and $L_p$-spaces,
$1 \leq p < \infty, p \neq 2$, any generalized bi-circular
projection $P$ is given by $P = \frac{I+T}{2}$, where $I$ is the
identity operator of the space and $T$ is a reflection, that is, $T$
is a surjective isometry with $T^2 = I$. For surjective isometries
of order $n \geq 3$, the corresponding notion of projection is
generalized $n$-circular projection as defined in \cite{AD}. In this
paper we show that in a Banach space $X$, if generalized bi-circular
projections are given by $\frac{I+T}{2}$ where $T$ is a reflection,
then any generalized $n$-circular projection $P$, $n \geq 3$, is
given by $P = \frac{I+T+T^2+\cdots+T^{n-1}}{n}$ where $T$ is a
surjective isometry and $T^n = I$. We prove our results for $n=3$
and for $n > 3$, the proof remains same except for routine
modifications.
\end{abstract}

\maketitle

\section{Introduction}
Let $X$ be a complex Banach space and $\bT$ denote the unit circle
in the complex plane. A projection $P$ on $X$ is said to be a
generalized bi-circular projection (hence forth GBP) if there exists
a $\la \in \bT \setminus \{1\}$ such that $P + \la(I-P)$ is a
surjective isometry on $X$. Here $I$ denotes the identity operator
on $X$.

It is easy to observe that any GBP is a bi-contractive projection.
It was proved in \cite{Lima} that any bi-contractive projection on
CL-spaces (which includes $C(\Omega)$ - $\Omega$ compact Hausdorff)
is given by $P = \frac{I+T}{2}$ where $I$ is the identity operator
of the space and $T$ is a reflection, that is, $T$ is a surjective
isometry of the space with $T^2 = I$.

Recently in a series of papers (see \cite{Bo, BJ1, BJ2, Fosner,
Lin}) it is observed that in many Banach spaces, the above holds true, that
is, a GBP of the space is always given by $\frac{I+T}{2}$ where $I$
is the identity operator and $T$ is a reflection. In particular,
this class of Banach spaces includes classical $L_p$-spaces, $1 \leq
p \leq \infty, p \neq 2$ and $C(\Omega, X)$ - the space of $X$
valued continuous functions on a compact Hausdorff space $\Omega$,
where $X$ is a Banach space such that vector valued Banach Stone
Theorem holds on $C(\Omega, X)$.

We note that in the case of GBP, if $P + \la (I-P)$ is a surjective
isometry and $\la \in \bT \setminus \{1\}$ is of infinite order then
$P$ is a hermitian projection (see \cite{Lin}). Such projections
were called trivial in \cite{DR, Lin}.

Suppose $X$ is a complex Banach space and $T$ is a surjective
isometry of $X$ such that $T^n = I, n \geq 2$. Suppose $P =
\frac{I+T+T^2+ T^{n-1}}{n}$ is a projection on $X$. Let $\la_0 =1,
\la_1, \la_2, \cdots, \la_{n-1}$ be the $n$ distinct roots of
identity. For $i = 1, 2, \cdots, n-1$, we define $P_i = \frac{I+
\ov{\la_i}T + \ov{\la_i}^2 T^2 + \cdots + \ov{\la_{i}}^{n-1}
T^{n-1}}{n}$. Then each $P_i$ is a projection, $P_0 \oplus P_1 \oplus
P_2 \oplus \cdots \oplus P_{n-1} = I$ and $P_0 + \la_1 P_1 + \la_2
P_2 + \cdots + \la_{n-1} P_{n-1} = T$.

For $n \geq 3$, we define

\bdfn \label{GNP} Let $X$ be a complex Banach space. A projection
$P_0$ on $X$ is said to be a generalized $n$-circular projection, $n
\geq 3$, if there exist $\la_1, \la_2, \cdots, \la_{n-1} \in \bT
\setminus \{1\}$, $\la_i,\ i= 1, 2, \cdots, n-1$ are of finite order
and projections $P_1, P_2, \cdots, P_{n-1}$ on $X$ such that \bla

\item $\la_i \neq \la_j$ for $i \neq j$

\item $P_0 \oplus P_1 \oplus \cdots \oplus P_{n-1} = I$

\item $P_0 + \la_1 P_1 + \cdots + \la_{n-1} P_{n-1}$ is a
surjective isometry. \el \edfn

\Rem In \cite{AD} generalized $n$-circular projection was defined
with an extra assumption that $i \neq j, i, j = 1, 2, \cdots, n-1$
then $\la_i \neq \pm \la_j$. It turns out for the validity of
results there and also in this paper, that assumption is not
necessary.
\end{rem}

The purpose of this note is to show that if in $X$ every GBP is
given by $\frac{I+T}{2}$ for reflection $T$, then for $n \geq 3$,
every generalized $n$-circular projection is given by
$\frac{I+T+T^2+ T^{n-1}}{n}$ where $T^n = I$. Precisely, we show

\begin{thm} \label{main}
Let $X$ be a complex Banach space. Suppose every GBP on $X$ is given
by $\frac{I+T}{2}$ where $T$ is a reflection. Let $P_0$ be a
generalized 3-circular projection on $X$. Then there exists an
surjective isometry $T$ on $X$ such that \bla

\item $P_0 + \omega P_1 + \omega^2 P_2 = T$ where $P_1$ and $P_2$
are as in Definition~\ref{GNP} and $\omega$ is a cube root of
identity,

\item $T^3 = I$. \el

Hence $P_0 = \frac{I + T + T^2}{3}$.
\end{thm}

\Rem \bla

\item The proof for the case $n > 3$ remains exactly same except for
number of cases to be considered in Lemma~\ref{lem3} in the next
section becomes larger.

\item In \cite{BJ2}, a GBP on $\ell_\infty$ was constructed which is
not given by average of identity and a surjective isometry of order
2. For generalized 3-circular projections, a similar example can
easily be constructed on $\ell_\infty$.\el
\end{rem}

\section{Proof of Theorem~\ref{main}}
We start with the following lemma.

\begin{lem} \label{lem1}
Let $X$ and $P_0$ be as in Theorem~\ref{main} and $P_1, P_2$ as in
Definition~\ref{GNP}. Then $\la_1$ and $\la_2$ are of same order.
\end{lem}

\begin{proof}
Let $\la^m_1 = \la^n_2 =1$ and $m \neq n$. Without loss of
generality we assume that $m <n$. Let $P_0 + \la_1 P_1 + \la_2 P_2 =
T$ where $T$ is a surjective isometry of $X$. Then $P_0 + \la^m_1
P_1 + \la^m_2 P_2 = (P_0 + P_1) + \la^m_2 P_2 = T^m$. Since $T^m$ is
again a surjective isometry and $P_2 = I - (P_0 + P_1)$, by the
assumption on $X$, $T$ is a reflection and hence we have $\la^m_2 =
-1$. Hence $n$ divides $2m$. Similarly we obtain $\la^n_1 = -1$ and
$m$ divides $2n$. Thus $2n = m k_1, 2m = n k_2$. Thus, $k_1 k_2 =
4$. Since we have assumed $m < n$, this implies $k_1 = 4, k_2 = 1$.
But then $-1 = \la^n_1 = \la^{2m}_1 = 1$ - a contradiction. Hence $m
=n$.
\end{proof}

The proof of the following lemma is straightforward verification and
hence we omit it.

\begin{lem} \label{lem2}
Let $X$ be a complex Banach space and $P_0$ a generalized 3-circular
projection on $X$. If $P_0 \oplus \la_1 P_1 \oplus \la_2 P_2 = T$
then $(T-\la_2I)(T-\la_1I)(T-I) =0$.
\end{lem}

For convenience of notation we write $T^* = S$ and $Q_i = P^*_i$ for
$i=0, 1, 2$. Note that $S$ is a surjective isometry of $X^*$ and
$Q_i,\ i = 0, 1, 2$ are projections on $X^*$ such that $Q_0 \oplus
Q_1 \oplus Q_2 = I$ - the identity operator on $X^*$, $Q_0 + \la_1
Q_1 + \la_2 Q_2 = S$. Also for any $n \geq 1$ we have $S^n = Q_0 +
\la^n_1 Q_1 + \la^n_2 Q_2$.

The following lemma is crucial in our proof.

\begin{lem} \label{lem3}
Let $X$ and $P_0$ be as in Theorem~\ref{main}. With above notation
we have the following. \bla

\item If for some $x^* \in X^*, x^* \neq 0, \ x^* = S x^*$ then $x^*
\in R(Q_0)$.

\item There is no $x^* \in X^*$ such that $x^* = S^2 x^*,\ x^* \neq
S x^*$.

\item If for some $x^* \in X^*, x^* \neq 0,\ x^* \neq S x^* \neq S^2
x^* \neq S^3 x^*$ then there exists $i = 1, 2$ such that $x^* \in
R(Q_i)$. \el
\end{lem}

\begin{proof}
\bla

\item Let $x^* \neq 0$ and $x^* = S x^*$. Then $x^*, S x^*, S^2 x^*$
are all equal hence we have,

\beqa x^* &=& Q_0 x^* + Q_1 x^* + Q_2 x^*\\
&=& Q_0 x^* + \la_1 Q_1 x^* + \la_2 Q_2 x^*\\
&=& Q_0 x^* + \la^2_1 Q_1 x^* + \la^2_2 Q_2 x^*. \eeqa

We choose $x \in X$ such that $x^*(x) \neq 0$. Let $Q_0 x^*(x) =
\al, Q_1 x^*(x) = \be, Q_2 x^*(x) = \ga$. The above equations give

\beqa \al + \be + \ga &=&1\\
&=& \al + \la_1 \be + \la_2 \ga\\
&=& \al + \la^2_1 \be + \la^2_2 \ga. \eeqa

Solving which we get $ \be = 0$ and $\ga = 0$ and $\al =1$. Thus
$x^* = Q_0 x^*$ and the assertion is proved.

\item Suppose there exists $x^*$ such that $x^* = S^2 x^*$ and $x^*
\neq S x^*$. In this case we have $S^3 x^* = S x^*$. We choose $x \in
X$ such that $x^*(x) = 1 = S^2 x^*(x)$ and $Sx^*(x) = S^3 x^* (x)
=0$. From Lemma~\ref{lem2} we know $(S-\la_2I)(S-\la_1I)(S-I)x^*
=0$. Evaluating this at $x$ we get $-1 -\la_1 - \la_2 = \la_1 \la_2$
and hence $(\la_1+1)(\la_2+1) = 0$ or $\la_1 = -1$ or $\la_2 = -1$.
If $\la_1 = -1$ then by Lemma~\ref{lem1} we get $\la_2 = 1 \rm{ or }
-1$. By our assumption $\la_2 \neq 1$ and if $\la_2 = -1$ then
$\la_1 = \la_2$ - a contradiction again.

\item From Lemma~\ref{lem2} it follows that $S^3 x^* \in \rm{span}
\{x^*, S x^*, S^2 x^*\}$ for all $x^* \in X^*$. Also from
Lemma~\ref{lem1} we have $S^n = I$ for some finite $n$ and hence
$x^* \in \rm{span} \{S x^*, S^2 x^*\}$ for all $x^* \in X^*$. Thus
given a $x^*$ we can write

$$x^* = \al S x^* + \be S^2 x^*$$
$$S^3 x^* = \al' S x^* + \be' S^2 x^*$$.

We claim if any of $\al, \al', \be, \be'$ equals $0$ then $x^* \in
R(Q_i)$ for one of $i = 1, 2$. To see this, we first observe that if
any of $\al'$ and $\be$ equals $0$ then $x^*$ is a multiple of
$S x^*$ and if $\al$ or $\be'$ is zero then $x^*$ is a multiple of
$S^2 x^*$. Let $x^* = \ga S x^*$ for some $\ga$. Thus we have
$x^* = Q_0 x^* + Q_1 x^* + Q_2 x^* = \ga Q_0 x^* + \ga \la_1 Q_1 x^*
+ \ga \la_2 Q_2 x^*$. Hence $(1-\ga)Q_0 x^* + (1-\ga \la_1) Q_1 x^*
+ (1-\ga \la_2)Q_2 x^* =0$. Now if $Q_0 x^* \neq 0$ we get $\ga =1$
contradicting the assumption that $x^* \neq S x^*$. Hence let $Q_0
x^* = 0$. If both $Q_1 x^*$ and $Q_2 x^*$ are non zero then we get
$\ga \la_1 = \ga \la_2 = 1$ and hence $\la_1 = \la_2$ a
contradiction again. Thus either $Q_1 x^* = 0$ or $Q_2 x^* =0$ and
$x^* = Q_i x^*$ for one of $i= 1, 2$. Similarly if $x^*$ is a multiple of
$S^2 x^*$ we proceed in the same way and use part (b)
to show that $x^* = Q_i x^*$ for one of $i= 1, 2$.

To conclude the proof let $\al, \al', \be, \be'$ are all no zero.
Since $S$ is invertible from the second equality above we obtain
$S^2 x^* = \al' x^* + \be' S x^*$. By the first part, if $Sx^*$ and
$S^2 x^*$ are multiple of each other then we get $x^* \in R(Q_i)$ for
one of $i=1, 2$. Thus let us assume $S x^*$ and $S^2 x^*$ are linearly independent.
Hence we get $\al = - \be \be'$ and $\al' = \frac{1}{\be}$. Hence we have
$x^* = Q_0 x^* + Q_1 x^* + Q_2 x^* = - \be \be'(Q_0x^* + \la_1 Q_1 x^* + \la_2 Q_2 x^*)
+ \be (Q_0 x^* + \la^2_1 Q_1 x^* + \la^2_1 Q_2 x^* = \be(1 - \be
\be')Q_0 x^* + \la_1\be(\la_1 - \be')Q_1 x^* + \la_2\be(\la_2 - \be')
Q_2 x^*$. If $Q_0 x^* = 0$ and $Q_1 x^* \neq 0$ and
$Q_2 x^* \neq 0$ we get $\la_1 = \la_2$ which contradicts our
assumption. Similarly, if $Q_0 x^* \neq 0$ and one of $Q_1 x^*$ and
$Q_2 x^*$ is zero we get $\be' = 1 = \la_1$ or
$\be' = 1 = \la_2$, a contradiction again.

Hence we may assume $Q_i x^* \neq 0$ for $i=0, 1, 2$.
This gives $\be - \be \be' =1, \la_1 \be (\la_1 - \be') =1$ and
$\la_2 \be (\la_2 - \be') =1$.

Similarly $S^3 x^* = \be S x^* - \be' x^*$ gives $\frac{1}{\be} +
\be' = 1, \la^2_1 = (\frac{1}{\be} + \la_1 \be')$ and $\la^2_2 =
(\frac{1}{\be} + \la_2 \be')$. The last two equation implies $(\la_1
- \la_2)(\la_1 + \la_2) = (\la_1 - \la_2)\be'$ and since $\la_1 \neq
\la_2$ we get $\la_1 + \la_2 = \be'$. Thus $\frac{1}\be = 1 - (\la_1
+ \la_2)$. Putting this values in $\la_1 \be (\la_1 - \be') =1$ we
get $\frac{-\la_1 \la_2}{1-\la_1 - \la_2} =1$ or $\la_1 = \la_2 =1$
which contradicts the assumptions on $\la_1, \la_2$. This completes
the proof of part $(b)$.

This completes the proof of the Lemma. \el
\end{proof}

{\em Completion of the proof of Theorem~\ref{main}:} By
Lemma~\ref{lem3} we can conclude that there exists $x^* \neq 0$ such
that $x^* = S^3 x^*$ and $x^* \neq S x^* \neq S^2 x^*$. If $Q_1 x^*
= Q_2 x^* = 0$ then $x^* \in R(Q_0)$. Thus let us assume $Q_i x^*
\neq 0$ for either $i=1,2$. But then $Q_0 x^* + Q_1 x^* + Q_2 x^* =
Q_0 x^* + \la^3_1 Q_1 x^* + \la^3_2 Q_2 x^*$ will imply either
$\la_1$ or $\la_2$ is a cube root of unity and hence by
Lemma~\ref{lem1} we get same for the other.

Thus $T = P_0 + \omega P_1 + \omega^2 P_2$ where $\omega$ is a cube
root of unity. This immediately gives $P_0 = \frac{I+T^2+T^3}{3}$.

\end{document}